\documentclass[11pt]{amsart}

\usepackage{lmodern} 
\usepackage[T1]{fontenc}

\theoremstyle{plain}
\newtheorem{theorem}                 {Theorem}      [section]
\newtheorem{proposition}  [theorem]  {Proposition}
\newtheorem{corollary}    [theorem]  {Corollary}

\theoremstyle{definition}

\newtheorem{remark}       [theorem]  {Remark}
\newtheorem{defi}         [theorem]  {Definition}

\numberwithin{equation}{section}

\def \mcf{f}
\def \mcv{H}

\DeclareMathOperator{\trace}{trace}
\DeclareMathOperator{\grad}{grad}

\DeclareMathOperator{\Id}{Id}

\DeclareMathOperator{\spn}{span}
\DeclareMathOperator{\degree}{degree}

\makeindex

\begin{document}

\title[Biharmonic submanifolds in $\mathbb{S}^{n}$]
{Biharmonic PNMC submanifolds in spheres}

\date{\today}

\author{A.~Balmu\c s}
\address{Faculty of Mathematics, ``Al.I.~Cuza'' University of Iasi\\
\newline
Bd. Carol I Nr. 11 \\
700506 Iasi, ROMANIA}
\email{adina.balmus@uaic.ro}
\author{S. Montaldo}
\address{Universit\`a degli Studi di Cagliari\\
Dipartimento di Matematica\\
\newline
Via Ospedale 72\\
09124 Cagliari, ITALIA} \email{montaldo@unica.it}
\author{C. Oniciuc}
\address{Faculty of Mathematics, ``Al.I.~Cuza'' University of Iasi\\
\newline
Bd. Carol I Nr. 11 \\
700506 Iasi, ROMANIA} \email{oniciucc@uaic.ro}


\thanks{The first author was supported by Grant POSDRU/89/1.5/S/49944, Romania. The second author was supported by Contributo d'Ateneo, University of Cagliari, Italy. The third author was supported by Grant PN-II-RU-TE-2011-3-0108, Romania.}

\begin{abstract}
We obtain several rigidity results for biharmonic submanifolds in
$\mathbb{S}^{n}$ with parallel normalized mean curvature vector
field. We classify biharmonic submanifolds in $\mathbb{S}^{n}$ with
parallel normalized mean curvature vector field and with at most two
distinct principal curvatures. In particular, we determine all
biharmonic surfaces with parallel normalized mean curvature vector
field in $\mathbb{S}^n$.

Then we investigate, for (not necessarily compact) proper biharmonic
submanifolds in $\mathbb{S}^n$, their type in the sense of
B-Y.~Chen. We prove: (i) a proper biharmonic submanifold in
$\mathbb{S}^n$ is of $1$-type or $2$-type if and only if it has
constant mean curvature ${\mcf}=1$ or ${\mcf}\in(0,1)$,
respectively; (ii) there are no proper biharmonic $3$-type
submanifolds with parallel normalized mean curvature vector field in
$\mathbb{S}^n$.
\end{abstract}


\maketitle

\section{Introduction}

Let $\varphi:M\to (N,h)$ be the inclusion of a submanifold $M$ into a
Riemannian manifold $(N,h)$. We say that the inclusion is {\em
biharmonic}, or $M$ is biharmonic, if its mean curvature vector
field $\mcv$ satisfies the following equation
\begin{eqnarray}\label{eq: bih_eq}
\tau_2(\varphi)=- m\left(\Delta \mcv + \trace{R^{N}}(
d\varphi(\cdot), \mcv) d\varphi(\cdot)\right)=0,
\end{eqnarray}
where $\Delta$ denotes the rough Laplacian on sections of the
pull-back bundle $\varphi^{-1}(TN)$ and $R^N$ denotes the curvature
operator on $(N,h)$. The section $\tau_2(\varphi)$ is called the
{\em bitension field}.

When $M$ is compact, the biharmonic condition arises from a
variational problem for maps: for an arbitrary smooth map
$\varphi:(M,g)\to (N,h)$ we define
$$
E_{2}\left( \varphi \right) = \frac{1}{2} \int_{M} |\tau(\varphi)|^{2}\, v_{g},
$$
where $\tau(\varphi)=\trace\nabla d\varphi$ is the {\it tension field}. The functional $E_2$ is called  the {\em bienergy functional}. In the particular case when $\varphi:(M,g)\to (N,h)$ is a Riemannian immersion, the tension
field has the expression $\tau(\varphi)=m\mcv$ and equation \eqref{eq: bih_eq} is equivalent to $\varphi$ being a critical point of $E_2$.

Obviously, any minimal submanifold ($\mcv=0$) is biharmonic. The non-harmonic biharmonic submanifolds  are called {\it proper biharmonic}.

The study of proper biharmonic submanifolds is nowadays becoming a
very active subject and its popularity initiated with the
challenging conjecture of B-Y.~Chen: {\em any biharmonic submanifold
in an Euclidean space is minimal}.

Due to some nonexistence results (see \cite{J86, O02}) the Chen
conjecture was generalized to: {\em any biharmonic submanifold in a
Riemannian manifold with nonpositive sectional curvature is
minimal}, but this was proved not to hold. Indeed, in \cite{OT10}
the authors constructed examples of proper biharmonic hypersurfaces
in a $5$-dimensional space of non-constant negative sectional
curvature.

Yet, the conjecture is still open in its full generality for ambient
spaces with constant nonpositive sectional curvature, although it
was proved to be true in numerous cases when additional geometric
properties for the submanifolds were assumed (see, for example,
\cite{BMO08,CMO02,C91,D92,HV95}).

By way of contrast, as we shall detail in Section~\ref{sec:
bih-sub}, there are several families of examples of proper
biharmonic submanifolds in the $n$-dimensional unit Euclidean sphere
$\mathbb{S}^{n}$. For simplicity we shall denote these classes by
B1, B2, B3 and B4. Nevertheless, a full understanding of the
geometry of proper biharmonic submanifolds in $\mathbb{S}^n$ has not
been achieved. The goal of this paper is to continue the study of
proper biharmonic submanifolds in $\mathbb{S}^{n}$ that was
initiated for the very first time in \cite{J86} and then developed
in \cite{BMO10} -- \cite{CMO01}, \cite{NU11, O02}.

In \cite{BO11} the proper biharmonic submanifolds with parallel mean
curvature vector field (PMC) in $\mathbb{S}^n$ were studied. In the
first part of this paper we extend our study to biharmonic
submanifolds with parallel normalized mean curvature vector field
(PNMC). We recall that there exist PNMC surfaces which are not PMC
(see \cite{C80, L}) and, obviously, a PNMC submanifold is PMC if and
only if it has constant mean curvature (CMC). We underline the fact
that all known examples of proper biharmonic submanifolds in spheres
are CMC, but there is no general result concerning the constancy of
the mean curvature of proper biharmonic submanifolds in
$\mathbb{S}^n$.

First, in Section~\ref{sec: pnmc}, under some hypotheses on the mean
curvature or on the squared norm of the Weingarten operator
associated to the mean curvature vector field, we prove that
compact, or complete, PNMC biharmonic submanifolds are PMC.

As we shall see in Section~\ref{sec: pnmc_<2}, PNMC pseudo-umbilical
biharmonic submanifolds in $\mathbb{S}^{n}$ are of class B3. We then
study the PNMC  biharmonic submanifolds in $\mathbb{S}^{n}$ with at
most two distinct principal curvatures in the direction of the mean
curvature vector field, proving that they are CMC and belong to the
classes B3 or B4 (Theorem~\ref{th: class_PNMC_AH_2}).

The second part of the paper is devoted to finite type submanifolds.
These submanifolds were introduced by B-Y.~Chen (see, for example,
\cite{C96,C84}) in the attempt of finding the best possible estimate
of the total mean curvature of a compact submanifold in the
Euclidean space. Although defined in a different manner, finite type
submanifolds arise also, in a natural way, as solutions of a
variational problem.

We prove that proper biharmonic submanifolds in spheres are of $1$-type or $2$-type if and only if they are CMC with mean curvature ${\mcf}=1$ or ${\mcf}\in(0,1)$, respectively (Theorem~ \ref{th: type12bih}).

Moreover, we prove that there are no $3$-type PNMC biharmonic submanifolds in $\mathbb{S}^{n}$ (Theorem~\ref{th: type3bih}),
obtaining  the nonexistence of $3$-type biharmonic hypersurfaces in
$\mathbb{S}^{n}$ (Corollary~\ref{cor: type3bih_hypersurf}).

Finally, under some extra conditions (mass-symmetric and independent) on finite $k$-type submanifolds in $\mathbb{S}^n$  we prove that biharmonicity implies that $k=2$  (Proposition~\ref{prop: high_type_bih}).

\vspace{2mm}

{\bf Conventions.}
Throughout this paper all manifolds, metrics, maps are assumed to be smooth, i.e. $C^\infty$. All manifolds are assumed to be connected. The following sign conventions are used
$$
\Delta V=-\trace\nabla^2 V\,,\qquad R^N(X,Y)=[\nabla_X,\nabla_Y]-\nabla_{[X,Y]},
$$
$V\in C(\varphi^{-1}(TN))$ and $X,Y\in C(TN)$.

\vspace{2mm}

{\bf Acknowledgements.}
The authors would like to thank Professors B-Y.~Chen and I.~Dimitric for helpful discussions.

\section{Biharmonic submanifolds}\label{sec: bih-sub}

The key ingredient in the study of biharmonic submanifolds is the
splitting of the bitension field with respect to its normal and
tangent components.

\begin{theorem}\label{th: bih subm N}
The canonical inclusion $\varphi:M^m\to N^n$ of a submanifold $M$ in a Riemannian manifold $N$ is biharmonic if and only if the normal and the tangent components of $\tau_2(\varphi)$ vanish, i.e. respectively
\begin{subequations}
\begin{equation}\label{eq: caract_bih_normal}
\Delta^\perp {\mcv}+\trace B(\cdot,A_{\mcv}\cdot)+\trace(R^N(\cdot,{\mcv})\cdot)^\perp=0,
\end{equation}
and
\begin{eqnarray}\label{eq: caract_bih_tangent}
\frac{m}{2}\grad {\mcf}^2+2\trace A_{\nabla^\perp_{(\cdot)}{\mcv}}(\cdot)+2\trace(R^N(\cdot,{\mcv})\cdot)^{\top}&=&\nonumber\\
-\frac{m}{2}\grad {\mcf}^2+2\trace (\nabla A_{\mcv})(\cdot,\cdot)&=&0,
\end{eqnarray}
\end{subequations}
where $A$ denotes the Weingarten operator, $B$ the second
fundamental form, ${\mcv}$ the mean curvature vector field, $\mcf=|\mcv|$ the mean curvature function,
$\nabla^\perp$ and $\Delta^\perp$ the connection and the Laplacian
in the normal bundle of $M$ in $N$.
\end{theorem}

This result was obtained in \cite{C84, O02} for submanifolds in space forms, and in \cite{O10} for general hypersurfaces. We
note that the tangent part of $\tau_2(\varphi)$ vanishes if and only
if the stress-energy tensor for biharmonic maps associated to
$\varphi$ vanishes (see \cite{LMO08, J87}). In the case the ambient
space is a space form of sectional curvature $c$, \eqref{eq:
caract_bih_normal} -- \eqref{eq: caract_bih_tangent} reduce to
\begin{corollary}[\cite{C84, O02}]\label{th: bih subm S^n}
The canonical inclusion $\varphi:M^m\to\mathbb{E}^n(c)$ of a submanifold $M$ in the space form $\mathbb{E}^n(c)$ is biharmonic if and only if
\begin{equation}\label{eq: caract_bih_spheres}
\left\{
\begin{array}{l}
\ \Delta^\perp {\mcv}+\trace B(\cdot,A_{\mcv}\cdot)-mc\,{\mcv}=0,
\vspace{2mm}
\\
\ 2\trace A_{\nabla^\perp_{(\cdot)}{\mcv}}(\cdot)
+\frac{m}{2}\grad {\mcf}^2=0.
\end{array}
\right.
\end{equation}
\end{corollary}

Up to now there are not known examples of proper biharmonic
submanifolds in a space form $\mathbb{E}^n(c)$ with $c\leq 0$, i.e.
proper solutions of \eqref{eq: caract_bih_spheres} with $c\leq 0$.
This fact has suggested, as we have mentioned in the introduction,
the  generalized Chen conjecture.

If $c=1$, the situation is rather different and the following are considered  to be the {\em main examples} of proper biharmonic submanifolds in $\mathbb{S}^n=\mathbb{E}^n(1)$:

\begin{list}{\labelitemi}{\leftmargin=2em\itemsep=1.5mm\topsep=2mm}
\item[{\bf B1}.] The small hypersphere
\begin{equation*}\label{eq: small_hypersphere}
\mathbb{S}^{n-1}(1/\sqrt 2)=\left\{(x,1/\sqrt 2)\in\mathbb{R}^{n+1}: x\in \mathbb{R}^n, |x|^2=1/2\right\}\subset\mathbb{S}^{n}.
\end{equation*}
\item[{\bf B2}.] The standard products of spheres
\begin{equation*}\label{eq: product_spheres}
\mathbb{S}^{n_1}(1/\sqrt 2)\times\mathbb{S}^{n_2}(1/\sqrt 2)=\left\{(x,y)\in\mathbb{R}^{n_1+1}\times\mathbb{R}^{n_2+1}, |x|^2=|y|^2=1/2\right\}\subset\mathbb{S}^{n},
\end{equation*}
$n_1+n_2=n-1$ and $n_1\neq n_2$.
\item[{\bf B3}.] The minimal submanifols $M$ in a small hypersphere $\mathbb{S}^{n-1}(1/\sqrt 2)\subset\mathbb{S}^n$.
\item[{\bf B4}.] The minimal submanifolds $M_1^{m_1}\times M_2^{m_2}$ in $\mathbb{S}^{n_1}(1/\sqrt 2)\times\mathbb{S}^{n_2}(1/\sqrt 2)\subset\mathbb{S}^n$, with $n_1+n_2=n-1$, $m_1\neq m_2$.
\end{list}

Example B2 was found in  \cite{J86}, while example B1 was derived in \cite{CMO01}. The two families of examples described in B3 and B4 were constructed in \cite{CMO02}. Moreover, B3 is a consequence of the following property.

\begin{theorem}[\cite{CMO02}]\label{th: rm_minim}
Let $M$ be a minimal submanifold in a small hypersphere $\mathbb{S}^{n-1}(a)\subset\mathbb{S}^n$, $a\in(0,1)$. Then $M$ is proper biharmonic in $\mathbb{S}^{n}$ if and only if $a=1/\sqrt 2$.
\end{theorem}


We note that the proper biharmonic submanifolds in $\mathbb{S}^n$
obtained from minimal submanifolds of the proper biharmonic
hypersphere $\mathbb{S}^{n-1}(1/\sqrt 2)$ have constant mean curvature ${\mcf}=1$.

More generally, we have the following bounds for the mean curvature  of  CMC proper  biharmonic submanifolds in $\mathbb{S}^n$.

\begin{theorem}[\cite{O03}]\label{th: classif_bih const mean}
Let $M$ be a CMC proper biharmonic submanifold in $\mathbb{S}^n$. Then ${\mcf}\in(0,1]$. Moreover, if
${\mcf}=1$, then $M$ is a minimal submanifold of a small hypersphere
$\mathbb{S}^{n-1}(1/\sqrt{2})\subset\mathbb{S}^n$.
\end{theorem}

Notice also that proper biharmonic submanifolds in $\mathbb{S}^n$
obtained from minimal submanifolds of $\mathbb{S}^{n-1}(1/\sqrt 2)$
have parallel mean curvature vector field (PMC) and are
pseudo-umbilical, i.e. $A_{\mcv}={\mcf}^2\Id$. In \cite{OW11} it was
proved that an umbilical biharmonic surface in any $3$-dimensional
Riemannian manifolds must be a CMC surface. This is a particular
case of the following proposition.

\begin{proposition}\label{th: bih_pseudo->CMC}
Let $\varphi:M^m\to N$ be a submanifold of a given space $N$, $m\neq 4$. If $M$ is pseudo-umbilical, then the tangent part of  $\tau_2(\varphi)$ vanishes, i.e. \eqref{eq: caract_bih_tangent} is satified, if and only if $M$ is CMC. In particular, if $M$ is a biharmonic pseudo-umbilical submanifold
of $N$, $m\neq 4$, then $M$ is CMC.
\end{proposition}
\begin{proof}
Since $M$ is pseudo-umbilical, $A_{\mcv}={\mcf}^2 \Id$ and we find immediately
\begin{equation}\label{eq-trace-A-pseudo-umbilical}
\trace (\nabla A_{\mcv})(\cdot,\cdot)=\grad {\mcf}^2.
\end{equation}
Then  \eqref{eq: caract_bih_tangent} is equivalent to
$$
(m-4) \grad {\mcf}^2=0,
$$
and we conclude the proof.
\end{proof}


We recall that a pseudo-umbilical submanifold $M^m$, $m\neq 4$, of
codimension two in $\mathbb{S}^{n}$ is proper biharmonic if and only
if it is minimal in $\mathbb{S}^{m+1}(1/\sqrt 2)$ (see
\cite{BMO08}). Now, a natural question arises: {\it for arbitrary
codimension, is a pseudo-umbilical proper biharmonic submanifold
$M^m$ in $\mathbb{S}^n$, $m\neq 4$, minimal in
$\mathbb{S}^{n-1}(1/\sqrt 2)$}?

\section{Biharmonic submanifolds with parallel normalized mean curvature vector field in $\mathbb{S}^n$}\label{sec: pnmc}

A submanifold $M$ in a Riemannian manifold is said to have \textit{parallel normalized mean curvature vector field} (PNMC) if it has nowhere zero mean curvature and the unit vector field in the direction of the mean curvature vector field is parallel in the normal bundle, i.e.
\begin{equation}\label{eq:def-pnmc}
\nabla^{\perp} ({{\mcv}}/{{\mcf}})=0,
\end{equation}
where $\mcf=|\mcv|$ is a smooth and positive function. In the following, for a PNMC submanifold, we shall denote by $\xi=H/f$ the nomalized mean curvature vector field and by $A$ the Weingarten operator associated to $\xi$.

PNMC submanifolds generalize non-minimal PMC submanifolds. Moreover, for CMC submanifolds PNMC is equivalent to PMC. Note that, as stated in \cite{C80, L}, it is possible to find examples of PNMC submanifolds which are not PMC.

The characterization of PNMC biharmonic submanifolds  in $\mathbb{S}^n$ follows by Corollary~\ref{th: bih subm S^n}.
\begin{theorem}\label{th: caract_bih_pnmc}
Let $\varphi:M^m\to\mathbb{S}^n$ be a PNMC submanifold in the $n$-dimensional unit Euclidean sphere $\mathbb{S}^n$. Then $M$ is biharmonic if and only if

\begin{equation}\label{eq: caract_bih_PNMC_spheres_1}
\begin{cases}
 \trace B(\cdot,A_{\mcv}\cdot)=\left(m-\dfrac{1}{{\mcf}}\,\Delta {\mcf}\right)\, {\mcv},
\vspace{2mm}\\
A_{\mcv}(\grad {\mcf}^2)=-\dfrac{m}{2}\,{\mcf}^2\grad {\mcf}^2,
\end{cases}
\end{equation}
or, equivalently,
\begin{equation}\label{eq: caract_bih_PNMC_spheres_2}
\begin{cases}
{\rm (i)}\quad \langle A, A_\eta\rangle=0,\qquad\qquad\forall\, \eta\in C(NM), \eta\perp \xi,
\vspace{2mm}\\
 {\rm (ii)}\quad \Delta {\mcf}=(m-|A|^2)\mcf,
\vspace{2mm}\\
{\rm (iii)}\quad A(\grad {\mcf}^2)=-\dfrac{m}{2}{\mcf}\grad {\mcf}^2,
\end{cases}
\end{equation}
where $NM$ denotes the normal bundle of $M$ in $\mathbb{S}^n$.
\end{theorem}

\begin{proof}

Let $p\in M$ and consider $\{E_i\}_{i=1}^m$ to be a local orthonormal frame field on $M$ geodesic at $p$.
Since $M$ is PNMC, we have
\begin{equation}\label{eq: nablaH_PNMC}
\nabla^{\perp}_X {\mcv}=\frac{1}{{\mcf}}\,X({\mcf}){\mcv},\qquad \forall\, X\in C(TM).
\end{equation}
From here, at $p$ we have
\begin{eqnarray*}
\Delta^\perp {\mcv}&=&-\trace (\nabla^\perp)^2 {\mcv}=-\sum_{i=1}^m\nabla^\perp_{E_i}\left( \frac{1}{{\mcf}}E_i({\mcf}){\mcv}\right)\\
&=&\frac{1}{{\mcf}}(\Delta {\mcf}) {\mcv},
\end{eqnarray*}
which implies that the first equation of \eqref{eq: caract_bih_spheres} becomes the first equation of \eqref{eq: caract_bih_PNMC_spheres_1}.

From $\eqref{eq: nablaH_PNMC}$ we obtain
\begin{equation}\label{eq:trace-A-nabla-perp}
\trace A_{\nabla^\perp_{(\cdot)}{\mcv}}(\cdot)=\sum_{i=1}^m A_{\nabla^\perp_{E_i}{\mcv}}(E_i)=\frac{1}{2{\mcf}^2}A_{\mcv}(\grad {\mcf}^2),
\end{equation}
and the second equation of \eqref{eq: caract_bih_spheres} becomes the second equation of \eqref{eq: caract_bih_PNMC_spheres_1}.

Next, since $A_H=\mcf A$, by considering the components of $\trace B(\cdot,A_{\mcv}\cdot)$, the one parallel to $\xi$ and the one orthogonal to $\xi$, one verifies immediately that equations \eqref{eq: caract_bih_PNMC_spheres_2} and \eqref{eq: caract_bih_PNMC_spheres_1} are equivalent.
\end{proof}

\subsection{The compact case}
Immediate consequences for compact PNMC biharmonic submanifolds
follow from \eqref{eq: caract_bih_PNMC_spheres_2}(ii).

\begin{corollary}\label{cor: comp_PNMC}
Let $M$ be a compact PNMC biharmonic submanifold in $\mathbb{S}^n$.
\begin{itemize}
\item[(i)]  If $|A|^2\leq m$, or $|A|^2\geq m$, on $M$, then $M$ is PMC
and $|A|^2=m$.

\item[(ii)] If $|A|$ is constant, then $M$ is PMC and $|A|^2=m$.
\end{itemize}
\end{corollary}

From Corollary~\ref{cor: comp_PNMC}, if $M$ is a compact PNMC
biharmonic submanifold in $\mathbb{S}^n$, then either there exists
$p\in M$ such that $|A(p)|^2<m$, or $|A|^2=m$.

Moreover, as a consequence of Corollary~\ref{cor: comp_PNMC}, we shall also prove that compact PNMC biharmonic submanifolds in $\mathbb{S}^n$, with a supplementary bounding condition on the mean curvature, are PMC. First we recall that a compact proper biarmonic submanifold in $\mathbb{S}^n$ admits at least one point $p$ with $\mcf(p)\leq 1$ (see \cite{BO11}), thus when considering the hypothesis $\mcf^2\geq \dfrac{4}{m}$ we have to assume $m\geq 5$ (if $m=4$, then $\mcf=1$).

\begin{proposition}\label{prop: bound_mc}
Let $M^m$ be a compact PNMC biharmonic submanifold in
$\mathbb{S}^n$, $m\geq 5$. If the mean curvature of $M$ satisfies
$\mcf^2\geq \dfrac{4}{m}$, then $M$ is PMC.
\end{proposition}

\begin{proof}
We will show that, in the given hypotheses, we have $|A|^2\geq m$ on $M$, thus, by Corollary \ref{cor: comp_PNMC}, $M$ is PMC.

Let $p_0\in M$ be arbitrarily fixed. We have two cases.

{\it Case 1}. If $\grad_{p_0}\mcf\neq 0$, since $M$ is PNMC biharmonic, from \eqref{eq: caract_bih_PNMC_spheres_2}(iii) we have
\begin{equation}\label{eq: A_H(p_0)}
|A(p_0)|^2\geq\frac{m^2}{4}\mcf^2(p_0)\geq m.
\end{equation}

{\it Case 2}. Consider now the case when $\grad_{p_0}\mcf=0$. If there exists an open subset $U\subset M$, $p_0\in U$, such that $\grad\mcf_{/U}=0$, then equation \eqref{eq: caract_bih_PNMC_spheres_2}(ii) implies that $|A|^2=m$ on $U$. Otherwise, $p_0$ is a limit point for the set $V=\{p\in M: \grad_p\mcf\neq 0\}$. By Case 1 we have $|A(p)|\geq m$, for all $p\in V$. Therefore, we obtain $|A(p_0)|^2\geq m$, and the proof is completed.
\end{proof}

By Proposition~\ref{prop: bound_mc} and Theorem 3.11 in \cite{BO11}
we get the following.
\begin{corollary}
Let $M^m$ be a compact PNMC biharmonic submanifold in $\mathbb{S}^n$
such that $\mcf^2\geq \dfrac{4}{m}$, $m\in \{4, 5, 6, 7\}$. Then $M$
is minimal in $\mathbb{S}^{n-1}(1/\sqrt 2)$.
\end{corollary}

Since hypersurfaces with nowhere zero mean curvature are PNMC submanifolds, we have the following result.

\begin{corollary}\label{prop: bound_hypersurf_mc}
Let $M^m$ be a compact biharmonic hypersurface in $\mathbb{S}^{m+1}$, $m\geq 5$. If the mean curvature of $M$ satisfies $\mcf^2\geq \dfrac{4}{m}$, then $M$ is CMC.
\end{corollary}

\subsection{The non-compact case} For the non-compact case, if $M^m$
is a PNMC biharmonic submanifold in $\mathbb{S}^n$ such that
$|A|^2\geq m$, then $\mcf$ is a subharmonic function and therefore
either $\mcf$ is constant, or $\mcf$ can not attain its maximum. In
the following we shall prove that, under some additional hypotheses,
the latter case can not occur.

We shall need the following theorem.
\begin{theorem}[Omori-Yau Maximum Principle, \cite{Y75}]\label{th: Omori-Yau}
If $M^m$ is a complete Riemannian manifold with Ricci curvature bounded from below, then for any function $u\in C^2(M)$, bounded from above, there exists a sequence of points $\{p_k\}_{k\in \mathbb{N}}\subset M$ satisfying
$$
\lim_{k\to\infty} u(p_k)=\sup_{M} u, \qquad |\grad_{p_k} u|<\frac{1}{k},\qquad \Delta u(p_k)>-\frac{1}{k}.
$$
\end{theorem}
Now we can prove our result.
\begin{proposition}\label{prop: cons_OY}
Let $M^m$ be a complete PNMC biharmonic submanifold with
non-negative Ricci curvature in $\mathbb{S}^n$. If $|A|$ is constant
and the mean curvature of $M$ satisfies $\mcf^2\geq \dfrac{4}{m}$,
then $M$ is PMC and $|A|^2=m$.
\end{proposition}
\begin{proof}
By an argument similar to that in the proof of
Proposition~\ref{prop: bound_mc}, we get ${|A|^2\geq m}$ on $M$, and
\eqref{eq: caract_bih_PNMC_spheres_2}(ii) implies $\Delta\mcf\leq 0$
on $M$.

On the other hand, we are in the hypotheses of Theorem~\ref{th:
Omori-Yau}, and thus there exists $\{p_k\}_k\subset M$ satisfying $
\Delta \mcf(p_k)>-\dfrac{1}{k}$. Therefore,
$\displaystyle{\lim_{k\to\infty}} \Delta\mcf(p_k)=0$. Using this in
\eqref{eq: caract_bih_PNMC_spheres_2}(ii), since $\mcf^2\geq
\dfrac{4}{m}$, we obtain $|A|^2=m$ and $\Delta \mcf=0$ on $M$.

Now, using the fact that on a complete manifold with non-negative
Ricci curvature there are no non-constant bounded harmonic functions
(see \cite{Y75}), we conclude.
\end{proof}

For hypersurfaces this result is expressed as follows.

\begin{corollary}\label{cor: Ricci_non-neg}
Let $M^m$ be a complete biharmonic hypersurface with non-negative
Ricci curvature in $\mathbb{S}^{m+1}$. If $|A|$ is constant and the
mean curvature of $M$ satisfies $\mcf^2\geq \dfrac{4}{m}$, then $M$
is CMC and $|A|^2=m$. In this case, $m\geq 8$ and $\mcf^2<1$.
\end{corollary}

\section{PNMC biharmonic submanifolds in $\mathbb{S}^n$ with at
most two distinct principal curvatures}\label{sec: pnmc_<2}

Inspired by the case of hypersurfaces (see \cite{BMO08}), we intend
to study PNMC biharmonic submanifolds in $\mathbb{S}^n$ by taking
into account the number of distinct principal curvatures in the
direction of the mean curvature vector field.

\begin{proposition}\label{pro-pseudo-pnmc}
Let $M^m$, $m\geq 2$, be a pseudo-umbilical PNMC submanifold in $\mathbb{S}^{n}$, then $M$ is PMC. Moreover, $M$ is minimal in $\mathbb{S}^{n-1}(a)\subset\mathbb{S}^{n}$, for some $a\in(0,1)$.
\end{proposition}
\begin{proof}
For any submanifold in $\mathbb{S}^{n}$ by the Codazzi equation, we have
\begin{equation}\label{eq:trace-A-nabla-perp-grad}
2\trace (\nabla A_{\mcv})(\cdot,\cdot)=m \grad {\mcf}^2+2\trace A_{\nabla^\perp_{(\cdot)}{\mcv}}(\cdot)
\end{equation}
Now, taking into account  \eqref{eq-trace-A-pseudo-umbilical} and \eqref{eq:trace-A-nabla-perp},
\eqref{eq:trace-A-nabla-perp-grad} becomes
$$
(m-1)\grad \mcf^2=0.
$$
Thus $M$ is PMC and, using  a result of B-Y.~Chen (see \cite[pag. 133]{C84}), follows that $M$ is minimal in $\mathbb{S}^{n-1}(a)\subset\mathbb{S}^{n}$, for some $a\in(0,1)$.
\end{proof}

Combining Proposition~\ref{pro-pseudo-pnmc} and  Theorem~\ref{th:
rm_minim}, follows that
\begin{proposition}\label{re: pnmc-pseudo}
Any PNMC pseudo-umbilical biharmonic submanifold  in $\mathbb{S}^n$
is minimal in $\mathbb{S}^{n-1}(1/\sqrt 2)$.
\end{proposition}

Thus, the next step consists in classifying  the PNMC biharmonic submanifolds in $\mathbb{S}^n$ with at most two distinct principal curvatures in the direction of ${\mcv}$. Notice that any hypersurface with nowhere zero mean curvature is PNMC, and the classification of proper biharmonic hypersurfaces with at most two distinct principal curvatures was achieved in \cite{BMO08}. In order to obtain the desired general classification, we first have to prove the following result.

\begin{theorem}\label{th: class_PNMC_AH_2->PMC}
Let $M^m$ be a PNMC biharmonic submanifold in $\mathbb{S}^n$ with at
most two distinct principal curvatures in the direction of ${\mcv}$.
Then $M$ is PMC.
\end{theorem}

\begin{proof}
It is sufficient to prove that $\mcf$, which is a positive function on $M$, is constant. Suppose that $\mcf\neq{\rm constant}$.
Then, there exists $p\in M$ such that $\grad_p \mcf\neq 0$, thus there exists $U$ a neighborhood of $p$ in $M$ such that $\grad \mcf\neq 0$ on $U$.
Taking into account Proposition~\ref{re: pnmc-pseudo},  $U$ can not be made out only of pseudo-umbilical points.
We can then assume that there exists a point $q\in U$ which is not pseudo-umbilical.  Then, eventually by restricting $U$, we can assume that $A\neq\mcf\Id$ at every point of $U$, thus $A$ has exactly two distinct principal curvatures on $U$.

From \eqref{eq: caract_bih_PNMC_spheres_2} (iii) we have
$$
A(\grad\mcf)=-\frac{m}{2}\mcf\grad\mcf,
$$
i.e. $E_1=\frac{\grad \mcf}{|\grad\mcf|}$ is a principal direction for $A$, with principal curvature
\begin{equation}\label{eq: k1}
k_1=-\frac{m}{2}\mcf.
\end{equation}

Recall that, as $A$ has exactly two distinct principal curvatures, the multiplicities of its principal curvatures are constant and the principal curvatures are smooth (see \cite{R69}). Thus $A$ is diagonalizable with respect to a local orthonormal frame field, and we can chose it to have the first field equal to $E_1$, i.e.  the frame field  is $\{E_1,\ldots, E_m\}$. We then have $A(E_i)=\bar{k}_i E_i$, $i=1,\ldots,m$, where not all the
$\bar{k}_i$'s are different and, by construction, $\bar{k}_1=k_1$. Since $\langle E_\alpha,E_1\rangle=0$, we have on $U$
\begin{equation}\label{eq: E_alpha_mcf}
E_\alpha(\mcf)=0\,, \quad \forall \, \alpha=2,\dots, m.
\end{equation}

We shall use the connection equations with respect to the frame field $\{E_1,\ldots, E_m\}$,
\begin{equation}\label{eq: conn_eq}
\nabla_{E_i}E_j=\omega_j^k(E_i)E_k.
\end{equation}

Let us first prove that the multiplicity of $k_1$ is $m_1=1$. Suppose that there exists $\alpha\in\{2,\ldots, m\}$, such that $\bar{k}_{\alpha}=k_1$ on $U$. Since $\nabla^\perp\xi=0$, the Codazzi equation for $A$ is
\begin{equation}\label{eq: codazzi_Am1}
(\nabla_{E_i} A)(E_j)=(\nabla_{E_j}A)(E_i),\quad \forall\, i,j=1,\ldots,m.
\end{equation}
By using \eqref{eq: conn_eq}, the Codazzi equation becomes
\begin{equation}\label{eq: codazzi_Am1_2}
E_i(\bar{k}_j)E_j+\sum_{\ell=1}^m(\bar{k}_j-\bar{k}_\ell)\omega_j^\ell(E_i)E_\ell= E_j(\bar{k}_i)E_i+\sum_{\ell=1}^m(\bar{k}_i-\bar{k}_\ell)\omega_i^\ell(E_j)E_\ell.
\end{equation}

Putting $i=1$ and $j=\alpha$ in \eqref{eq: codazzi_Am1_2} and taking the scalar product with $E_\alpha$ we obtain $E_1(k_1)=0$, which, together with \eqref{eq: k1} and \eqref{eq: E_alpha_mcf}, gives $\mcf={\rm constant}$, thus we have a contradiction.

Thus $\bar{k}_1=k_1$ and $\bar{k}_\alpha=k_2$, for all $\alpha=2,\ldots,m$,
and since $\trace A=m\mcf$, we get
\begin{equation}\label{eq: k2}
k_2=\frac{3}{2}\frac{m}{m-1}\mcf.
\end{equation}

Putting  $i=1$ and $j=\alpha$ in \eqref{eq: codazzi_Am1_2} and taking the scalar product with $E_\alpha$, $E_\beta$, $\beta\neq\alpha$, and $E_1$, respectively, one gets
\begin{subequations}
\begin{equation}\label{eq: cons_c_1}
\omega_1^\alpha(E_\alpha)=-\frac{3}{m+2}\frac{E_1(\mcf)}{\mcf},
\end{equation}
\begin{equation}\label{eq: cons_c_2}
\omega_1^\alpha(E_\beta)=0,
\end{equation}
\begin{equation}\label{eq: cons_c_3}
\omega_1^\alpha(E_1)=0,
\end{equation}
for all $\alpha, \beta=2,\ldots,m$, $\alpha\neq\beta$.
\end{subequations}


Consider $\{\eta_{m+1}=\xi,\eta_{m+2}\ldots,\eta_n\}$ to be an orthonormal normal frame field on $U$ in $\mathbb{S}^n$ and denote by $A_a=A_{\eta_a}$, $a=m+2,\ldots, n$. Since $\nabla^\perp \xi=0$, from the Ricci equation of $U$ in $\mathbb{S}^n$, we have
$$
A\circ A_a=A_a\circ A,\quad \forall\, a=m+2,\ldots, n.
$$
Since  $k_1$ has multiplicity $1$, if follows directly that  $E_1$ is a principal direction for $A_a$, for all $a=m+2,\ldots,n$. Fix $a\in\{m+2,\ldots,n\}$ and denote $A_a(E_1)=\lambda_a E_1$ on $U$. From \eqref{eq: caract_bih_PNMC_spheres_2} (i), we have that $\sum_{i=1}^m \langle A(E_i),A_a(E_i)\rangle=0$ and this leads to
$$
(k_1-k_2)\lambda_a+k_2\trace A_a=0.
$$
Since $\trace A_a=m\langle \mcv, \eta_a \rangle=0$, we conclude that $\lambda_a=0$, i.e.
\begin{equation}\label{eq: AaE1}
A_a(E_1)=0,\quad \forall \,a=m+2,\ldots,n.
\end{equation}

We now express the Gauss equation for $U$ in $\mathbb{S}^n$,
\begin{eqnarray}\label{eq: Gauss_EQ}
\langle R^{\mathbb{S}^n}(X,Y)Z,W\rangle&=&\langle
R(X,Y)Z,W\rangle\nonumber\\&&+\langle B(X,Z),B(Y,W)\rangle-\langle
B(X,W),B(Y,Z)\rangle,
\end{eqnarray}
with  $X=W=E_1$ and $Y=Z=E_\alpha$. Using \eqref{eq: AaE1} one obtains
$$
B(E_1,E_\alpha)=0,\quad
B(E_1,E_1)=k_1\xi,\quad
\langle B(E_\alpha,E_\alpha), B(E_1,E_1)\rangle=k_1 k_2.
$$
From \eqref{eq: conn_eq}, \eqref{eq: cons_c_2}, \eqref{eq: cons_c_3}, and using $\omega^k_j=-\omega^j_k$, the curvature term is
$$
\langle R(E_1,E_\alpha)E_\alpha,E_1\rangle=-E_1(\omega_1^\alpha(E_\alpha)) -(\omega^\alpha_1(E_\alpha))^2.
$$
Finally, \eqref{eq: Gauss_EQ} and \eqref{eq: cons_c_1} imply
\begin{equation}\label{eq: cons_Gauss}
\mcf E_1(E_1(\mcf))=\frac{m+2}{3}\mcf^2 -\frac{m^2(m+2)}{4(m-1)}\mcf^4+\frac{m+5}{m+2}(E_1(\mcf))^2.
\end{equation}
From \eqref{eq: k1} and \eqref{eq: k2}, we have
\begin{equation}\label{eq: normA}
|A|^2=k_1^2+(m-1)k_2^2=\frac{m^2(m+8)}{4(m-1)}\mcf^2.
\end{equation}
Moreover, using \eqref{eq: E_alpha_mcf}, \eqref{eq: conn_eq} and \eqref{eq: cons_c_1} the Laplacian of $\mcf$ becomes
\begin{eqnarray}\label{eq: Delta}
\Delta\mcf&=& -E_1(E_1(\mcf))-\sum_{\alpha=2}^m E_\alpha(E_\alpha(\mcf))+(\nabla_{E_1}E_1)\mcf+\sum_{\alpha=2}^m (\nabla_{E_\alpha} E_\alpha)\mcf\nonumber\\
&=&-E_1(E_1(\mcf))+\sum_{\alpha=2}^m \omega_\alpha^1(E_\alpha)E_1(\mcf)\nonumber\\
&=&-E_1(E_1(\mcf))+\frac{3(m-1)}{m+2}\frac{(E_1(\mcf))^2}{f}.
\end{eqnarray}

From \eqref{eq: caract_bih_PNMC_spheres_2} (ii), by substituting \eqref{eq: normA} and \eqref{eq: Delta}, we get
\begin{equation}\label{eq: cons_bih}
\mcf E_1(E_1(\mcf))=-m\mcf^2 +\frac{m^2(m+8)}{4(m-1)}\mcf^4+\frac{3(m-1)}{m+2}(E_1(\mcf))^2.
\end{equation}

Consider now $\gamma=\gamma(u)$ to be an arbitrary integral curve of $E_1$ in $U$. Along $\gamma$ we have $f=f(u)$ and we set $w=(E_1(f))^2=(\mcf')^2$. Then ${dw}/{d\mcf}=2\mcf''$, and \eqref{eq: cons_Gauss} and \eqref{eq: cons_bih} become
\begin{equation}\label{eq: sist_mcf}
\left\{
  \begin{array}{ll}
    \dfrac{1}{2}\mcf \dfrac{dw}{d\mcf}=\dfrac{m+2}{3}\mcf^2-\dfrac{m^2(m+2)}{4(m-1)}\mcf^4+\dfrac{m+5}{m+2}w, \vspace{2mm}\\
\dfrac{1}{2}\mcf \dfrac{dw}{d\mcf}=-m\mcf^2+\dfrac{m^2(m+8)}{4(m-1)}\mcf^4+\dfrac{3(m-1)}{m+2}w.
  \end{array}
\right.
\end{equation}
By subtracting the two equations we find two cases.

If $m=4$, then
$$
-\frac{2(2m+1)}{3\mcf^2}f^2+\frac{m^2(m+5)}{2(m-1)}\mcf^4=0,
$$
thus $\mcf$ is constant.

If $m\neq 4$, then
$$
w=\frac{(m+2)(2m+1)}{3(m-4)}\mcf^2-\frac{m^2(m+2)(m+5)}{4(m-4)(m-1)}\mcf^4.
$$
Differentiating with respect to $f$ and replacing this in the second equation of \eqref{eq: sist_mcf}, we get
$$
\frac{(m-1)(m+5)}{3}\mcf^2+\frac{3m^2(2m+1)}{4(m-1)}\mcf^4=0.
$$
Therefore $\mcf$ is constant along $\gamma$, thus $\grad \mcf=0$ along $\gamma$ and we have a contradiction.
\end{proof}

As a consequence of Theorem~\ref{th: class_PNMC_AH_2->PMC} we have the following rigidity result.

\begin{theorem}\label{th: class_PNMC_AH_2}
Let $M^m$ be a PNMC biharmonic submanifold in $\mathbb{S}^n$ with at
most two distinct principal curvatures in the direction of ${\mcv}$.
Then either $M$ is minimal in $\mathbb{S}^{n-1}(1/\sqrt 2)$, or
locally,
$$
M=M^{m_1}_1\times M^{m_2}_2\subset \mathbb{S}^{n_1}
(1/\sqrt 2)\times \mathbb{S}^{n_2}(1/\sqrt 2)\subset\mathbb{S}^n,
$$
where $M_i$ is a minimal submanifold of $\mathbb{S}^{n_i}(1/\sqrt
2)$, $i=1,2$, $m_1+m_2=m$, $m_1\neq m_2$, $n_1+n_2=n-1$.
\end{theorem}
\begin{proof}
From Theorem~\ref{th: class_PNMC_AH_2->PMC} we conclude that $M$ is PMC. Moreover, since $A_H$ has at most two distinct principal curvatures in the direction of ${\mcv}$, from Proposition~3.19 in \cite{BO11}, we get that $\nabla A_H=0$ and the conclusion follows by applying Theorem~3.16 in \cite{BO11}.
\end{proof}

Moreover, as a corollary of Theorem \ref{th: class_PNMC_AH_2}, the following rigidity result, which generalizes Theorem~5.6 in \cite{BMO08}, is valid.
\begin{corollary}\label{th: class_surf_bih_PNMC}
Let $M^2$ be a PNMC biharmonic surface in $\mathbb{S}^n$. Then $M$
is minimal in $\mathbb{S}^{n-1}(1/\sqrt 2)$.
\end{corollary}

\begin{remark}

(i) In \cite{C80} it was proved that, in general, a PNMC analytic surface in $\mathbb{S}^n$ is either minimal in a small hypersphere of $\mathbb{S}^n$, and therefore it is PMC, or it lies in a $4$-dimensional great sphere $\mathbb{S}^4\subset\mathbb{S}^n$. Notice that with no analyticity condition, by Corollary \ref{th: class_surf_bih_PNMC}, the supplementary hypothesis that the surface is biharmonic leads only to the first case.

(ii) For the particular case of PNMC biharmonic surfaces in
$\mathbb{S}^4$ we can give a different proof for Theorem  \ref{th:
class_PNMC_AH_2}. Indeed, using the codimension reduction result of
Erbacher (\cite{E71}), one can prove that the surface lies in a
great hypersphere $\mathbb{S}^3$ of $\mathbb{S}^4$ and, therefore it
must have constant mean curvature.

(iii) We can slightly relax the hypotheses of Theorem \ref{th: class_PNMC_AH_2}, obtaining the same result, in the following way. By the unique continuation property for biharmonic maps (see \cite{O03}), if $M$ is a proper biharmonic submanifold in $\mathbb{S}^n$, then $\mcv$ is nowhere zero on an open dense subset $W\subset M$. If we assume that $\nabla^{\perp}(\mcv/\mcf)=0$ on $W$ and $A_\mcv$ has at most two distinct principal curvatures everywhere on $W$, then by Theorem \ref{th: class_PNMC_AH_2} we get $\nabla^{\perp}\mcv=0$ on $W$. By continuity we obtain $\nabla^{\perp}\mcv=0$ on $M$.

\end{remark}

\section{On the type of biharmonic submanifolds in $\mathbb{S}^n$}\label{sec: fine-type}

\begin{defi}[\cite{C96, C84}]\label{def: finite type}
A submanifold $M\subset\mathbb{R}^{n+1}$ is called of \textit{finite type} if its inclusion $\phi:M\to \mathbb{R}^{n+1}$  can be expressed as a finite sum of $\mathbb{R}^{n+1}$-valued eigenmaps of the Laplacian $\Delta$ of $M$, i.e.
\begin{equation}\label{eq: spec_decomp}
\phi=\phi_0+\phi_{t_1}+\ldots+\phi_{t_k},
\end{equation}
where $\phi_0\in \mathbb{R}^{n+1}$ is a constant vector, $\phi_{t_i}:M\to \mathbb{R}^{n+1}$ are non-constant maps satisfying $\Delta\phi_{t_i}=\lambda_{t_i}\phi_{t_i}$, $i=1,\ldots, k$. If, in particular, all eigenvalues $\lambda_{t_i}$ are assumed to be mutually distinct the submanifold is said to be of \textit{k-type} and \eqref{eq: spec_decomp} is called the \textit{spectral decomposition} of $\phi$.
\end{defi}
\begin{remark}
If $M$ is compact the inclusion $\phi:M\to \mathbb{R}^{n+1}$ admits a unique spectral decomposition
$\phi=\phi_0+\sum_{i=1}^{\infty}\phi_{i}$, where $\phi_0$ is the {\it center of mass}. Then, it is of $k$-type if only  $k$ terms of  $\{\phi_i\}_{i=1}^{\infty}$ are not vanishing.   In the non compact case the spectral decomposition  $\phi=\phi_0+\sum_{i=1}^{\infty}\phi_{i}$ is not guaranteed. Nonetheless, if Definition~\ref{def: finite type} is satisfied the spectral decomposition is unique.  Notice also that, in the non-compact case, the harmonic component of the spectral decomposition is not necessarily constant. Finite type submanifolds with non-constant harmonic component are called {\em null finite type} submanifolds.
\end{remark}

The inclusion of a \textit{k-type} submanifold is said to be \textit{linearly independent} if the linear subspaces
$$
E_{t_i}=\spn \{\phi_{t_i}(u): u\in M\},\quad i=1,\ldots, k,
$$
are linearly independent, i.e. the dimension of the subspace spanned by vectors in $\bigcup_{i=1}^k E_{t_i}$ is equal to $\sum_{i=1}^k\dim E_{t_i}$.

The following result provides us a necessary and a sufficient condition for a submanifold to be of finite type.

\begin{theorem}\cite{C84, CP91}\label{th: crit_fin_type}
Let $\phi:M\to \mathbb{R}^{n+1}$ be an isometric immersion.
\begin{itemize}
\item[(i)] If $M$ is of finite $k$-type, there exist a constant
vector $\phi_0\in \mathbb{R}^{n+1}$ and a monic polynomial with
simple roots $P$ of degree $k$ with $P(\Delta)(\phi-\phi_0)=0$ .
\item[(ii)] If there exist a constant vector $\phi_0\in \mathbb{R}^{n+1}$ and a polynomial $P$ with simple roots such that $P(\Delta)(\phi-\phi_0)=0$, then $M$ is of finite $k$-type with $k\leq \degree(P)$.
\end{itemize}
\end{theorem}

We shall also use the following version.

\begin{theorem}\cite{C84, CP91}\label{th: crit_fin_type_H}
Let $\phi:M\to \mathbb{R}^{n+1}$ be an isometric immersion.
\begin{itemize}
\item[(i)] If $M$ is of finite $k$-type, there exists a monic polynomial $P$ of degree $k-1$ or $k$ with $P(\Delta)\mcv^0=0$ .
\item[(ii)] If there exists a polynomial $P$ with simple roots such that $P(\Delta)\mcv^0=0$, then $M$ is of infinite type or of finite $k$-type with $k-1\leq \degree(P)$.
\end{itemize}
Here ${\mcv}^0$ denotes the mean curvature vector field of $M$ in $\mathbb{R}^{n+1}$.
\end{theorem}

A well known result of T.~Takahashi can be rewritten as the classification of $1$-type submanifolds in $\mathbb{R}^{n+1}$.
\begin{theorem}[\cite{T66}]\label{th: Takahashi}
A submanifold $M\subset\mathbb{R}^{n+1}$ is of $1$-type if and only if either $M$ is a minimal submanifold of $\mathbb{R}^{n+1}$, or $M$ is a minimal submanifold of a hypersphere of $\mathbb{R}^{n+1}$.
\end{theorem}

\begin{defi}
A submanifold $M\subset\mathbb{S}^n$ is said to be of {\em finite type} if it is of finite type as a submanifold of $\mathbb{R}^{n+1}$, where $\mathbb{S}^n$ is canonically embedded in $\mathbb{R}^{n+1}$. Moreover, a non-null finite type submanifold in $\mathbb{S}^n$ is said to be \textit{mass-symmetric} if the constant vector $\phi_0$ of its spectral decomposition is the center of the hypersphere $\mathbb{S}^n$, i.e. $\phi_0=0$.
\end{defi}

\begin{remark}
By Theorem~\ref{th: Takahashi}, biharmonic submanifolds of class B3  are $1$-type submanifolds. Indeed, the inclusion $\phi:M\to \mathbb{R}^{n+1}$ of $M$ in $\mathbb{R}^{n+1}$ has the spectral decomposition
$$
\phi=\phi_0+\phi_p,
$$
where $\phi_0=(0,1/\sqrt 2)$, $\phi_p:M\to\mathbb{R}^{n+1}$, $\phi_p(x,1/\sqrt 2)=(x,0)$ and $\Delta \phi_p=2m \phi_p$.

Moreover, biharmonic submanifolds of class B4 are mass-symmetric $2$-type submanifolds. Indeed, $\phi:M_1\times M_2\to \mathbb{R}^{n+1}$ has the spectral decomposition
$$
\phi=\phi_p+\phi_q,
$$
where $\phi_p(x,y)=(x,0)$, $\phi_q(x,y)=(0,y)$, $\Delta \phi_p=2m_1\phi_p$, $\Delta \phi_q=2m_2\phi_q$.
\end{remark}

Let $\varphi:M\to\mathbb{S}^n$ be a submanifold in $\mathbb{S}^n$ and denote by $\phi=\mathbf{i}\circ\varphi:M\to\mathbb{R}^{n+1}$ the inclusion of $M$ in $\mathbb{R}^{n+1}$. Denote by ${\mcv}$ the mean curvature
vector field of $M$ in $\mathbb{S}^n$ and by ${\mcv}^0$ the mean curvature vector field of $M$ in $\mathbb{R}^{n+1}$.

The mean curvature vector fields ${\mcv}^0$ and ${\mcv}$ are
related by ${\mcv}^0={\mcv}-\phi$. Moreover, we have
\begin{equation}\label{eq: H0H}
\langle {\mcv},\phi\rangle=0,\quad \langle {\mcv}^0,{\mcv}\rangle={\mcf}^2,\quad \langle {\mcv}^0, \phi\rangle=-1.
\end{equation}

Following \cite{CMO02} the bitension field of $\varphi$ can be written as
$$
\tau_2(\varphi)=-m \Delta\mcv^0+2m^2\mcv^0+m^2\{2-|\mcv^0|^2\}\phi.
$$
Thus,
$\tau_2(\varphi)=0$ if and only if
\begin{subequations}
\begin{equation}\label{eq: caract_bih_HH}
\Delta {\mcv}^0-2m{\mcv}^0+m({\mcf}^2-1)\phi=0.
\end{equation}
or equivalently, since $\Delta\phi=-m\mcv^0$,
\begin{equation}\label{eq: caract_bih_varphi}
\Delta^2\phi-2m\Delta\phi-m^2({\mcf}^2-1)\phi=0,
\end{equation}
\end{subequations}

In \cite[Theorem~3.1]{BMO08} we proved that CMC compact proper biharmonic submanifolds in $\mathbb{S}^n$ are of $1$-type or $2$-type.  This result can be generalized to the following.
\begin{theorem}\label{th: type12bih}
Let $\varphi:M\to\mathbb{S}^n$ be a proper biharmonic submanifold, not necessarily compact, in the unit Euclidean sphere $\mathbb{S}^n$. Denote by $\phi=\mathbf{i}\circ\varphi:M\to\mathbb{R}^{n+1}$ the inclusion of $M$ in $\mathbb{R}^{n+1}$, where $\mathbf{i}:\mathbb{S}^n\to\mathbb{R}^{n+1}$ is the canonical inclusion map. Then
\begin{itemize}
  \item[(i)] $M$ is a $1$-type submanifold of $\mathbb{R}^{n+1}$ if and only if ${\mcf}=1$. In this case, $\phi=\phi_0+\phi_p$, $\Delta \phi_p=2m\phi_p$, and $\phi_0\in\mathbb{R}^{n+1}$, $|\phi_0|=1/\sqrt 2$.

  \item[(ii)] $M$ is a $2$-type submanifold if and only if ${\mcf}={\rm constant}$, ${\mcf}\in(0,1)$. In this case, $\phi=\phi_p+\phi_q$, $\Delta \phi_p=m(1-{\mcf})x_p$, $\Delta \phi_q=m(1+{\mcf})x_q$.
\end{itemize}
\end{theorem}

\begin{proof}
In order to prove (i), notice that the converse is obvious, by Theorem~\ref{th: Takahashi} and Theorem~\ref{th: classif_bih const mean}.

Let us suppose that $M$ is a $1$-type submanifold. From Theorem~ \ref{th: crit_fin_type_H}(i), it follows that there exists $a\in\mathbb{R}$ such that
\begin{equation}\label{eq: M_1-type}
\Delta {\mcv}^0=a{\mcv}^0.
\end{equation}
Equations \eqref{eq: caract_bih_HH} and \eqref{eq: M_1-type} imply
$$
(2m-a){\mcv}^0-m({\mcf}^2-1)\phi=0,
$$
and by considering the scalar product with ${\mcv}$ and using \eqref{eq: H0H}, since $M$ is proper biharmonic, we get $a=2m$ and
$$
m({\mcf}^2-1)\phi=0.
$$
Thus ${\mcf}=1$. Now, as the map $\phi$ can not be harmonic,
\eqref{eq: caract_bih_varphi} leads to the spectral decomposition
$\phi=\phi_0+\phi_p$, $\Delta\phi_p=2m\phi_p$. Since $\Delta
\phi=-m\mcv^0$, taking into account the relation between $\mcv$ and
$\mcv^0$, we obtain $2\phi_0=\phi+\mcv$. Since $|\phi|=1=\mcf$, and
$\mcv$ is orthogonal to $\phi$, we conclude that $|\phi_0|=1/\sqrt
2$.

Let us now prove (ii). The converse of (ii) follows immediately.
Indeed, from \eqref{eq: caract_bih_varphi}, if ${\mcf}={\rm
constant}$, ${\mcf}\in (0,1)$, then choosing the constant vector
$\phi_0=0$ and the polynomial with simple roots
$$
P(\Delta)=\Delta^2-2m\Delta^1-m^2({\mcf}^2-1)\Delta^0,
$$
we are in the hypotheses of Theorem~\ref{th: crit_fin_type}(ii). Thus $M$ is of finite $k$-type, with $k\leq2$. Taking into account (i), since ${\mcf}\in(0,1)$, this implies that $M$ is a $2$-type submanifold with $$
\phi=\phi_p+\phi_q,
$$
with corresponding eigenvalues $\lambda_p=m(1-{\mcf})$, $\lambda_q=m(1+{\mcf})$. Also, notice that
$$
\phi_p=\frac{\lambda_q}{\lambda_q-\lambda_p}\phi-\frac{1}{\lambda_q-\lambda_q}\Delta\phi,\qquad
\phi_q=-\frac{\lambda_p}{\lambda_q-\lambda_p}\phi+\frac{1}{\lambda_q-\lambda_q}\Delta\phi,
$$
which are smooth non-zero maps.

Suppose now that $M$ is a $2$-type submanifold. From Theorem~\ref{th: crit_fin_type}(i), it follows that there exist a constant vector $\phi_0\in\mathbb{R}^{n+1}$ and $a, b\in\mathbb{R}$ such that
\begin{equation}\label{eq: M_2-type}
\Delta {\mcv}^0=a{\mcv}^0+b(\phi-\phi_0).
\end{equation}
Equations \eqref{eq: caract_bih_HH} and \eqref{eq: M_2-type} lead to
\begin{equation}\label{eq: conseq_2+bih}
(2m-a){\mcv}^0-(m({\mcf}^2-1)+b)\phi+b \phi_0=0.
\end{equation}
We have to consider two cases.\\
\textit{Case 1.} If $b=0$, i.e. $M$ is a null $2$-type submanifold, by taking the scalar product with ${\mcv}$ in \eqref{eq: conseq_2+bih} and using \eqref{eq: H0H}, since $M$ is proper biharmonic, we get $a=2m$ and ${\mcf}=1$. By (i), this leads to a contradiction.\\
\textit{Case 2.} If $b\neq 0$, by taking the scalar product with $X\in C(TM)$ in \eqref{eq: conseq_2+bih}, we obtain $\langle \phi_0, X\rangle=0$, for all $X\in C(TM)$, i.e. the component of $\phi_0$ tangent to $M$ vanishes
\begin{equation}\label{eq: varphi0^T=0}
(\phi_0)^{\top}=0.
\end{equation}
Take now the scalar product with $\phi$ in \eqref{eq: conseq_2+bih} and use \eqref{eq: H0H}. We obtain
$$
-2m+a-m({\mcf}^2-1)-b+b\langle \phi_0,\phi\rangle=0,
$$
and, by differentiating,
\begin{equation}\label{eq: consec_prod_phi}
m\grad {\mcf}^2=b\grad \langle \phi_0, \phi\rangle.
\end{equation}
Now, by considering $\{E_i\}_{i=1}^m$ to be a local orthonormal frame field on $M$, we have
\begin{eqnarray}
\grad\langle \phi_0, \phi\rangle&=&\sum_{i=1}^mE_i(\langle \phi_0, \phi\rangle)E_i=\sum_{i=1}^m\langle \phi_0,\nabla^{0}_{E_i}\phi\rangle E_i=\sum_{i=1}^m\langle \phi_0,E_i\rangle E_i\nonumber\\
&=&(\phi_0)^{\top}.
\end{eqnarray}
This, together with equations \eqref{eq: varphi0^T=0} and \eqref{eq: consec_prod_phi}, leads to ${\mcf}={\rm constant}$ and using Theorem~\ref{th: classif_bih const mean} we conclude the proof.
\end{proof}

\begin{remark}\quad
The direct implication of (i) in Theorem~\ref{th: type12bih} can be also proved in a more geometric manner (see \cite{BFO09}).
\end{remark}

We are now interested in proper biharmonic submanifolds of $3$-type
in spheres. In \cite{CJ91} it was proved that there are no CMC
3-type hypersurfaces in a hypersphere of the Euclidean space. Since
the known examples of proper biharmonic hypersurfaces in spheres are
CMC, one may think that there are no such hypersurfaces of type
$3$. Indeed, we have a more general result.
\begin{theorem}\label{th: type3bih}
There exist no PNMC biharmonic $3$-type submanifolds $M^m$ in the unit Euclidean sphere $\mathbb{S}^n$.
\end{theorem}
\begin{proof}
Suppose now that $M$ is a PNMC biharmonic $3$-type submanifold. From Theorem~ \ref{th: crit_fin_type}, it follows that there exist a constant vector $\phi_0\in\mathbb{R}^{n+1}$ and $a, b, c\in\mathbb{R}$ such that
\begin{equation}\label{eq: M_3-type}
\Delta^2 {\mcv}^0=a\Delta {\mcv}^0+b{\mcv}^0+c(\phi-\phi_0).
\end{equation}
Equations \eqref{eq: caract_bih_HH} and \eqref{eq: M_3-type} lead to
\begin{equation}\label{eq: consec3type}
\Delta^2 {\mcv}^0=(2ma+b){\mcv}^0+(c-ma({\mcf}^2-1))\phi-c\phi_0.
\end{equation}
Now, by applying $\Delta$ to equation \eqref{eq: caract_bih_HH} we get
\begin{eqnarray}\label{eq: consecbih}
\Delta^2{\mcv}^0&=&m^2(3+{\mcf}^2){\mcv}^0-(m\Delta{\mcf}^2+2m^2({\mcf}^2-1))\phi\nonumber\\
&&+2m\,d\phi(\grad {\mcf}^2).
\end{eqnarray}
By taking the scalar product with $\xi={\mcv}/\mcf$ in \eqref{eq: consec3type} and \eqref{eq: consecbih} and by using \eqref{eq: H0H}, we obtain
\begin{equation}\label{eq: prod H}
-c\langle \phi_0,\xi\rangle=m^2{\mcf}^3+(3m^2-2ma-b){\mcf}.
\end{equation}
Consider now the scalar product with $X\in C(TM)$ in \eqref{eq: consec3type} and \eqref{eq: consecbih}. This implies
$$
-c\langle \phi_0,X\rangle=2m X({\mcf}^2),
$$
and, further, the component of $c\phi_0$ tangent to $M$ is given by
\begin{equation}\label{eq: prod Xi}
-c(\phi_0)^{\top}=2m \grad{\mcf}^2.
\end{equation}
Moreover, by taking the scalar product with an arbitrary vector field $\eta$ normal to $M$ in $\mathbb{S}^n$, $\eta\perp \xi$, in \eqref{eq: consec3type} and \eqref{eq: consecbih} we find
\begin{equation}\label{eq: prod eta}
-c\langle \phi_0,\eta\rangle=0.
\end{equation}
Equations \eqref{eq: prod Xi} and \eqref{eq: prod eta} lead to
\begin{equation}\label{eq: varphi0}
-c\phi_0=2m\grad{\mcf}^2-c\langle \phi_0,\xi\rangle \xi-c\langle \phi_0,\phi\rangle\phi.
\end{equation}

Differentiating \eqref{eq: prod H}, one gets
\begin{equation}\label{eq: grad_varphi0_H}
-c\grad\langle\phi_0,\xi\rangle=(3m^2{\mcf}^2+3m^2-2ma-b)\grad{\mcf}.
\end{equation}
By considering $\{E_i\}_{i=1}^m$ to be a local orthonormal frame field on $M$ and using $\nabla^\perp\xi=0$, \eqref{eq: varphi0}, \eqref{eq: prod H} and \eqref{eq: prod Xi}, we have the following
\begin{eqnarray}\label{eq: grad_varphi0_H_bis}
-c\grad\langle\phi_0,\xi\rangle&=&-c\sum_{i=1}^m E_i(\langle\phi_0,\xi\rangle)E_i=-c\sum_{i=1}^m (\langle \nabla^{0}_{E_i}\phi_0,\xi\rangle+\langle \phi_0,\nabla^{0}_{E_i}\xi\rangle)E_i\nonumber\\
&=&-c\sum_{i=1}^m \langle \phi_0,\nabla_{E_i}\xi\rangle E_i=-c\sum_{i=1}^m \langle \phi_0,\nabla^\perp_{E_i}\xi -A(E_i)\rangle E_i\nonumber\\
&=&c\sum_{i=1}^m \langle (\phi_0)^{\top}, A(E_i)\rangle E_i
=\sum_{i=1}^m \langle A(c(\phi_0)^{\top}), E_i\rangle E_i\nonumber\\
&=&-2mA(\grad {\mcf}^2).
\end{eqnarray}

Equations \eqref{eq: grad_varphi0_H} and \eqref{eq: grad_varphi0_H_bis} imply
\begin{equation}\label{eq: AH_gradH2_3}
2m A(\grad {\mcf}^2)=\left(-3m^2{\mcf}^2 -3m^2+2ma+b\right)\grad {\mcf}.
\end{equation}

Since $M$ is a PNMC biharmonic submanifold, \eqref{eq:
caract_bih_PNMC_spheres_2} (iii), together with \eqref{eq:
AH_gradH2_3}, leads to
$$
\left(m^2{\mcf}^2+3m^2-2ma-b\right)\grad {\mcf}=0,
$$
on $M$. This implies that $\grad \mcf=0$, i.e. $M$ is CMC. From Theorem~\ref{th: type12bih}, we have that $M$ is a $1$-type or $2$-type submanifold and we get to a contradiction.

\end{proof}

Since any hypersurface with nowhere zero mean curvature is PNMC we have the following.
\begin{corollary}\label{cor: type3bih_hypersurf}
There exist no biharmonic $3$-type hypersurfaces $M^m$ in the unit Euclidean sphere $\mathbb{S}^{m+1}$.
\end{corollary}
\begin{proof}
Suppose that $M$ is of $3$-type. Then $M$ is not minimal in
$\mathbb{S}^{m+1}$, thus $\mcf$ is nowhere zero on an open dense
subset $W\subset M$. Every connected component of $W$ is PNMC and,
by Theorem~\ref{th: type3bih}, it can not be of $3$-type. This leads
to a contradiction.
\end{proof}

We note that the classes B3 and B4 of proper biharmonic submanifolds in spheres are linearly independent (even more, orthogonal) $2$-type submanifolds. Thus it is natural to ask weather there exist proper biharmonic independent higher finite type submanifolds. We can prove the following result.
\begin{proposition}\label{prop: high_type_bih}
Let $M$ be a proper biharmonic submanifold in $\mathbb{S}^n$. If  $M$ is of finite $k$-type, mass-symmetric and linearly independent, then $k=2$.
\end{proposition}
\begin{proof}
Let $M$ be a $k$-type mass-symmetric submanifold in $\mathbb{S}^n$. Then, we have
$$
\phi=\phi_{t_1}+\phi_{t_2}+\ldots+\phi_{t_k},
$$
where $\phi_{t_i}$ are non-harmonic maps satisfying $\Delta\phi_{t_i}=\lambda_{t_i}\phi_{t_i}$, and $\lambda_{t_i}$ are mutually distinct, $i=1,\ldots, k$.
This implies that
\begin{equation}\label{eq: delta-delta2}
\Delta\phi=\sum_{i=1}^k \lambda_{t_i} \phi_{t_i}
\qquad\textrm{and}\qquad
\Delta^2\phi=\sum_{i=1}^k \lambda_{t_i}^2 \phi_{t_i}.
\end{equation}
Since $M$ is proper biharmonic, replacing \eqref{eq: delta-delta2} in \eqref{eq: caract_bih_varphi}, we obtain
$$
\sum_{i=1}^k (\lambda_{t_i}^2-2m\lambda_{t_i}-m^2({\mcf}^2-1)) \phi_{t_i}=0.
$$
Using that $M$ is independent, we get
$(\lambda_{t_i}^2-2m\lambda_{t_i}-m^2({\mcf}^2-1))\phi_{t_i}=0$ on
$M$, for all $i=1,\ldots,k$. Since $\phi_{t_i}$ is non-zero on an
open dense set in $M$, we have
$\lambda_{t_i}^2-2m\lambda_{t_i}-m^2({\mcf}^2-1)=0$ on $M$, for all
$i=1,\ldots,k$. This implies that ${\mcf}={\rm constant}$. Since
$\phi$ is mass-symmetric, by Theorem~\ref{th: type12bih}, we
conclude that $k=2$.
\end{proof}


\begin{thebibliography}{99}


\bibitem{BFO09} A. Balmu\c s, D.~Fetcu, C.~Oniciuc.
Stability properties for biharmonic maps. Geometry Exploratory Workshop on ``Differential Geometry and its Applications'' Ia\c si, September 2--4, 2009, 1--19, Cluj University Press, 2011.

\bibitem{BMO10} A. Balmu\c s, S.~Montaldo, C.~Oniciuc.
Biharmonic hypersurfaces in 4-dimensional space forms. Math. Nachr. 283 (2010), 1696--1705.

\bibitem{BMO08} A. Balmu\c s, S. Montaldo, C. Oniciuc.
Classification results for biharmonic submanifolds in spheres. Israel J. Math. 168 (2008), 201--220.

\bibitem{BO11} A. Balmu\c s, C. Oniciuc.
Biharmonic submanifolds with parallel mean
curvature vector field in spheres. J. Math. Anal. Appl. 386 (2012), 619--630.

\bibitem{BO09} A. Balmu\c s, C. Oniciuc.
Biharmonic surfaces of $\mathbb{S}^4$. Kyushu J. Math. 63 (2009), 339--345.

\bibitem{CMO02} R.~Caddeo, S.~Montaldo, C.~Oniciuc.
Biharmonic submanifolds in spheres. Israel J. Math. 130 (2002), 109--123.

\bibitem{CMO01}R.~Caddeo, S.~Montaldo, C.~Oniciuc.
Biharmonic submanifolds of $\mathbb{S}^3$. Internat. J. Math. 12 (2001), 867--876.



\bibitem{C96} B-Y.~Chen.
A report on submanifolds of finite type. Soochow J. Math. 22 (1996), 117--337.

\bibitem{C91} B-Y.~Chen.
Some open problems and conjectures on submanifolds of finite type. Soochow J. Math. 17 (1991), 169--188.

\bibitem{C84} B-Y.~Chen.
Total Mean Curvature and Submanifolds of Finite Type. Series in Pure Mathematics, 1. World Scientific Publishing Co., Singapore, 1984.



\bibitem{C80} B-Y.~Chen.
Surfaces with parallel normalized mean curvature vector. Mh. Math. 90 (1980), 185--194.

\bibitem{CJ91} B-Y.~Chen, S.J.~Li.
$3$-type hypersurfaces in a hypersphere. Bull. Soc. Math. Belg. S\'{e}r. B 43 (1991), no. 2, 135--141.

\bibitem{CP91} B-Y.~Chen, M.~Petrovic.
On spectral decomposition of immersions of finite type. Bull. Austral. Math. Soc. 44 (1991), 117--129.

\bibitem{D92} I.~Dimitric.
Submanifolds of $\mathbb{E}^m$ with harmonic mean curvature vector. Bull. Inst. Math. Acad. Sinica 20 (1992), 53--65.


\bibitem{ES64} J.~Eells, J.H.~Sampson.
Harmonic mappings of Riemannian manifolds. Amer. J. Math. 86 (1964), 109--160.

\bibitem{E71} J.~Erbacher.
Reduction of the codimension of an isometric immersion. J. Differ.
Geom. 5 (1971), 333--340.

\bibitem{HV95} T.~Hasanis, T.~Vlachos.
Hypersurfaces in $E^4$ with harmonic mean curvature vector field. Math. Nachr. 172 (1995), 145--169.


\bibitem{J87} G.Y.~Jiang.
The conservation law for $2$2-harmonic maps between Riemannian
manifolds. Acta Math. Sinica 30 (1987), 220--225.

\bibitem{J86} G.Y.~Jiang.
$2$-harmonic isometric immersions between Riemannian manifolds. Chinese Ann. Math. Ser. A 7 (1986),
130--144.

\bibitem{L} D.S.P.~Leung.
The Cauchy problem for surfaces with parallel normalized mean
curvature vector in a four dimensional Riemannian manifold.
preprint.

\bibitem{LM08} E.~Loubeau, S.~Montaldo.
Biminimal immersions. Proc. Edinb. Math. Soc. 51 (2008), 421--437.

\bibitem{LMO08} E.~Loubeau, S.~Montaldo, C.~Oniciuc.
The stress-energy tensor for biharmonic maps. Math. Z. 259 (2008), 503--524.


\bibitem{NU11} N.~Nakauchi, H.~Urakawa.
Biharmonic hypersurfaces in a Riemannian manifold with non-positive
Ricci curvature. Ann. Glob. Anal. Geom. 40 (2011), 125--131.

\bibitem{O03} C.~Oniciuc.
Tangency and Harmonicity Properties. PhD Thesis,
Geometry Balkan Press 2003,
\texttt{http://www.mathem.pub.ro/\-dgds/mono/dgdsmono.htm}

\bibitem{O02} C.~Oniciuc.
Biharmonic maps between Riemannian manifolds. An. Stiint. Univ. Al.I. Cuza Iasi Mat (N.S.) 48 (2002), 237--248.

\bibitem{O10} Y. -L.~Ou.
Biharmonic hypersurfaces in Riemannian manifolds. Pacific J. Math. 248 (2010), 217--232.


\bibitem{OT10}Y.-L.~Ou, L.~Tang.
The generalized Chen's conjecture on biharmonic submanifolds is false. \texttt{arXiv:1006.1838}.

\bibitem{OW11}Y.-L.~Ou, Z-P.~Wang.
Constant mean curvature and totally umbilical biharmonic surfaces in
$3$-dimensional geometries. J. Geom. Phys. (10) 61, (2011),
1845--1853.

\bibitem{R69} P.J.~Ryan.
Homogeneity and some curvature conditions for hypersurfaces.
T\^{o}hoku Math. J. (2) 21 (1969), 363--388.


\bibitem{S68} J.~Simons.
Minimal varieties in Riemannian manifolds. Ann. of Math. 88 (1968), 62--105.


\bibitem{T66} T.~Takahashi.
Minimal immersions of Riemannian manifolds. J. Math. Soc. Japan 18 (1966), 380--385.

\bibitem{Y75} S.-T.~Yau.
Harmonic functions on complete Riemannian manifolds. Commun. Pure Appl. Math. 28 (1975), 201--228.

\end{thebibliography}
\end{document}